\documentclass[preprint]{elsarticle}
\usepackage{slashbox}
\usepackage{amsmath,amsopn,amssymb,epsfig,amsthm}
\usepackage{multirow}
\usepackage{algorithm,algorithmic}
\usepackage{color}

\definecolor{amethyst}{rgb}{0.6, 0.4, 0.8}
\definecolor{orange}{rgb}{1,0.5,0}

\newtheorem{Theorem}{Theorem}[section]
\newtheorem{Lemma}{Lemma}[section]
\newtheorem{Proposition}{Proposition}[section]
\newtheorem{Remark}{Remark}[section]
\newtheorem{example}{Example}[section]

\newproof{pot}{Proof of Theorem \ref{thm2}}
\newcommand{\bb}{\begin{bmatrix}}
\newcommand{\eb}{\end{bmatrix}}
\newcommand{\bl}[1]{\begin{list}{#1}{\usecounter{bean}}} \newcommand{\el}{\end{list}}
\newcommand{\bel}[1]{\begin{equation} \label{#1}} \newcommand{\eel}{\end{equation}}

\def\r2n2n{\mathbb{R}^{2n\times 2n}}
\def\c2n2n{\mathbb{C}^{2n\times 2n}}

\begin{document}

\date{}
\begin{frontmatter}
\title{\color{black}Inheritance properties of the conjugate discrete-time algebraic Riccati equation}
%
\author{Chun-Yueh Chiang\corref{cor1}}
\ead{chiang@nfu.edu.tw}
\address{Center for General Education, National Formosa
University, Huwei 632, Taiwan.}
\author{Hung-Yuan Fan \corref{cor2}}
\ead{hyfan@ntnu.edu.tw}
\address{Department of Mathematics, National Taiwan Normal University, Taipei 116325, Taiwan.}

\cortext[cor2]{Corresponding author}

\date{ }

\begin{abstract}
In this paper we consider a class of conjugate discrete-time Riccati equations,
arising originally from the linear quadratic regulation problem for discrete-time antilinear systems.
Under mild and reasonable assumptions, the existence of the maximal solution to the conjugate discrete-time Riccati equation, in which the control
weighting matrix is nonsingular and its constant term is Hermitian, will be inherited to a transformed discrete-time algebraic Riccati equation.
Based on this inheritance property, an accelerated fixed-point iteration is proposed for finding the maximal solution via the transformed Riccati equation.
Numerical examples are shown to illustrate the correctness of our theoretical results and the feasibility of the proposed algorithm.
\end{abstract}

\begin{keyword}
conjugate discrete-time algebraic Riccati equation, inheritance property, fixed-point iteration,
maximal solution, LQR control problem,
antilinear system
\MSC 39B12\sep39B42\sep93A99\sep65H05\sep15A24
 \end{keyword}

\end{frontmatter}

\section{Introduction} \label{sec1}
In this paper we are mainly concerned with the inheritance properties of the maximal solution to the conjugate
discrete-time algebraic Riccati equation (CDARE) of the form
\begin{subequations}\label{cdare}
\begin{align}
  X = \mathcal{R}(X):= A^H \overline{X}A - A^H \overline{X}B(R+B^H \overline{X}B)^{-1} B^H \overline{X}A + H, \label{cdare-a}
\end{align}
{\color{black} or the compact form}
 \begin{align}
 X = A^H \overline{X} (I+ G\overline{X})^{-1} A + H, \label{cdare-b}
\end{align}
\end{subequations}
where $A\in \mathbb{C}^{n\times n}$, $B\in \mathbb{C}^{n\times m}$, $R\in \mathbb{H}_m$ is nonsingular,
$H\in \mathbb{H}_n$, $I$ is the identity matrix of compatible size and
$G:= BR^{-1}B^H$, respectively. Here, $\mathbb{H}_\ell$ denotes the set of all $\ell \times \ell$ Hermitian matrices and
$\overline{M}$ is the matrix obtained by taking the complex conjugate of each entry of a matrix $M$. The solution $X\in \mathbb{H}_n$ of the CDARE \eqref{cdare-a},
with $R_X:=R+B^H \overline{X}B$ being positive definite, is considered in this paper.

For any $M,N\in \mathbb{H}_n$, the positive definite and positive semidefinite matrices are denoted by $M > 0$ and $M\geq 0$, respectively.
Moreover, we usually write $M \geq N$ (or $M \leq N$) if $M - N \geq 0$ (or $N - M \geq 0$) in the context.
{ For the sake of simplicity, the spectrum and spectral radius of $A\in \mathbb{C}^{n\times n}$ are
denoted by $\sigma (A)$ and $\rho (A)$, respectively. We say that the CDARE \eqref{cdare} has a maximal solution $X_M\in \mathbb{H}_n$
if $X_M\geq X$ {for any solution $X\in \mathbb{H}_n$ of the CDARE.} Thus, it follows from the definition of the maximality that $X_M$ is unique if it exists.} A matrix operator $f:\mathbb{H}_{n}\rightarrow\mathbb{H}_{n}$ is order preserving (resp. reversing) on $\mathbb{H}_{n}$ if $f(A)\geq f(B)$ (resp.  $f(A)\leq f(B)$) when $A\geq B$ and $A,B\in\mathbb{H}_{n}$.


A class of CDAREs \eqref{cdare} arises originally from linear quadratic regulation (LQR) optimal control problem for the discrete-time antilinear system
\begin{equation} \label{antils}
  x_{k+1} = A\overline{x_k} + B\overline{u_k},\quad k\geq 0,
\end{equation}
where $x_k\in \mathbb{C}^n$ is the state response and $u_k\in \mathbb{C}^m$ is the control input.
 The main goal of this control problem is to find a state feedback control $\color{black} u_k = -\overline{F} x_k$ such that the performance index
 \[ \mathcal{J}(u_k, x_0) := \sum_{k=0}^\infty (x_k^H H x_k + u_k^H R u_k)  \]
is minimized with $H\geq 0$ and $R > 0$. If the antilinear system \eqref{antils} is controllable,
it is shown in Theorem 12.7 of \cite{w.z17} that the optimal state feedback controller is {\color{black}$u_k^* = -\overline{F_{X_*}} x_k$
with $F_{X_*}  := R_{X_*}^{-1}B^H \overline{X_*}A$} such that
the minimum value of $\mathcal{J}(u_k^*, x_0) = x_0^H X_* x_0$ is achieved and the corresponding closed-loop antilinear system
\begin{equation} \label{clm}
  x_{k+1} = (A - BF_{X_*}) \overline{x_k},\quad k\geq 0,
\end{equation}
is asymptotically stable, i.e., $\lim\limits_{k\rightarrow \infty} x_k = 0$,
where $X_* \geq 0$ is the unique solution of the CDARE \eqref{cdare-a} or the {\em discrete-time algebraic anti-Riccati equation} \cite{w.q.l.s16,w.z17}.

 Recently, some accelerated iterations have been proposed for solving the unique positive definite solution  of the CDARE \eqref{cdare} under positive definiteness assumptions with $G > 0$ and $H > 0$, see, e.g., \cite{l.c18,l.c20,m19} and the references therein. In addition, this numerical technique has also been utilized to some real-life applications recently,
see, e.g., \cite{l.w.g20,r.l.a20,r.l.e21}.

It is worth mentioning that the authors in \cite{l.c18} first transformed a conjugate nonlinear matrix equation, arising from modern quantum theory,
to some CDARE (1b) with $G, H > 0$, and proposed theoretical results for the existence and uniqueness of the (maximal) positive definite solution to the CDARE.
On the other hand, we have generalized the existence of the maximal solution for the CDARE \eqref{cdare} with $G, H\in \mathbb{H}_n$
under the framework of the fixed-point iteration (FPI) and some weaker assumptions \cite{f.c23}, see also Theorem~\ref{thm:Xmax} below.

Inspired by the technique using in \cite{l.c18}, if
\begin{align}\label{a1}
\det(R_H)\neq 0
\end{align}
, where $R_H := R + B^H \overline{H} B$, then the CDARE~\eqref{cdare-a}
can be transformed into a discrete-time algebraic Riccati equation (DARE) of the form
\begin{subequations} \label{dare2}
\begin{align}
X = \widehat{\mathcal{R}}(X) &:= \mathcal{R}(\mathcal{R}(X)) = \widehat{A}^H {X}\widehat{A} - \widehat{A}^H {X}\widehat{B}(\widehat{R}+
\widehat{B}^H {X}\widehat{B})^{-1} \widehat{B}^H \widehat{A} + \widehat{H} \label{dare2-a} \\
&= \widehat{A}^H X (I + \widehat{G} X)^{-1}\widehat{A} + \widehat{H}, \label{dare2-b}
\end{align}
\end{subequations}
where its coefficient matrices are given by
\begin{subequations}\label{coeff-dare2}
  \begin{align}
\widehat{A}&= \overline{A}A-\overline{A}B(R+B^H \overline{H} B)^{-1}B^H\overline{H}A,
\widehat{B} =\bb \overline{B} & \overline{A}B \eb, \label{hatA} \widehat{R} = \overline{R}\oplus R_H, \\
\widehat{G}&= \widehat{B} \widehat{R}^{-1} \widehat{B}^H = \overline{G}+\overline{A} (I+G\overline{H})^{-1}{G}\overline{A}^H,
\widehat{H} = H+{A}^H \overline{H}(I+G\overline{H})^{-1}{A}, \label{hatG}
\end{align}
\end{subequations}
and the invertibility of the matrix $\widehat{R}_X := \widehat{R} + \widehat{B}^H X\widehat{B}$ will be shown in the next section.
The assumption~\eqref{a1} will be needed throughout the paper.

 Therefore, the first question is to ask whether the CDARE \eqref{cdare} and its induced DARE \eqref{dare2}
both have the same maximal solution $X_M\in \mathbb{H}_n$ under the assumptions proposed in \cite{f.c23}. If yes, we may further ask if there exists
a suitable matrix $\widehat{F}_{X_M}\in \mathbb{C}^{2m\times n}$ related to $X_M$ so that the closed-loop matrix $\widehat{A} - \widehat{B}\widehat{F}_{X_M}$
inherits the stability of the closed-loop matrix $A-BF_{X_M}$ related to the
original CDARE \eqref{cdare}, where $F_{X_M}$ is defined by \eqref{clm}.
These interesting questions will be addressed completely in this work.

Once the CDARE \eqref{cdare} and its transformed DARE \eqref{dare2} have the same
maximal solution, there are dozens of numerical methods for solving the standard DAREs of small to medium sizes
 in the existing literatures, see, e.g., \cite{c.f.l.w04,g.k.l92,k88,l.r95,l79,p.l.s80,v81} and  the references therein.
Nevertheless, following the line of the FPI discussed in \cite{f.c23}, we are mainly interested in the design of
an accelerated fixed-point iteration (AFPI) for finding the maximal solution of
the CDARE \eqref{cdare} via its transformed DARE \eqref{dare2} directly.

The paper is organized as follows. In Section~\ref{sec2}, we introduce some useful notations, and auxiliary theorems and lemmas
that will be used in our main results. Some inheritance properties related to the CDARE \eqref{cdare} and its transformed
DARE \eqref{dare2} have been addressed in Section~\ref{sec:inheri}.
Especially, under reasonable and mild assumptions, we have shown that these Riccati matrix equations both have the
same maximal Hermitian solution in Theorem~\ref{thm:inheri-1}. Based on this theorem, an accelerated
fixed-point iteration is proposed for computing the maximal solution of the transformed DARE \eqref{dare2} directly in Section~\ref{sec:afpi}.
Numerical examples are given to illustrate the validity of the inheritance properties and the feasibility of our AFPI algorithm
in Section~\ref{sec:NumExamp}. Finally, we conclude the paper in Section~\ref{sec:cr}.

\section{Preliminaries}\label{sec2}

In this section we introduce some notations and auxiliary lemmas that will be used below. Firstly, let the conjugate Stein matrix operator $\mathcal{C}_A:\mathbb{H}_n\rightarrow\mathbb{H}_n$  and standard Stein matrix operator $\mathcal{S}_A:\mathbb{H}_n\rightarrow\mathbb{H}_n$
associated with a matrix $A\in\mathbb{C}^{n\times n}$ be defined by
\begin{equation}\label{SME}
\mathcal{C}_A(X):= X-A^H \overline{X} A,\quad \mathcal{S}_A(X):= X-A^H X A
\end{equation}
for any $X\in\mathbb{H}_n$. In general, the operator $\mathcal{C}_A$ is neither order preserving nor order reversing. However, under the assumption that $\rho(\overline{A}A)<1$, its inverse operator $\mathcal{C}^{-1}_A$ is order preserving,
since $\mathcal{C}_A^{-1}(X)=\sum\limits_{k=0}^\infty ((\overline{A}A)^k)^H (X+A^H \overline{X} A) (\overline{A}A)^k \geq
\sum\limits_{k=0}^\infty ((\overline{A}A)^k)^H (Y+A^H \overline{Y} A) (\overline{A}A)^k=\mathcal{C}_A^{-1}(Y)$
for $X,Y\in \mathbb{H}_n$ with $X\geq Y$.

{\color{black}
It will be clear later on that the matix operators $\mathcal{R}:\mathrm{dom}(\mathcal{R})\rightarrow \mathbb{H}_n$ and $\widehat{\mathcal{R}}:\mathrm{dom}(\widehat{\mathcal{R}})\rightarrow \mathbb{H}_n$ defined by \eqref{cdare-a} and \eqref{dare2-a}, respectively,
both play an important role in our theorems below, where $\mathrm{dom}(\mathcal{R}):=\{X\in\mathbb{H}_n\, |\, \det(R_X)\neq 0\}$ and
$\mathrm{dom}(\widehat{\mathcal{R}}):=\{X\in\mathbb{H}_n\, |\, \det(\widehat{R}_X)\neq 0\}$. If $R_X$ is nonsingular for $X\in \mathrm{dom}(\mathcal{R})$,
it is easily seen that the matrix $\widehat{R}_X$ defined by \eqref{hatA} is of the form
\begin{align}\label{hRx}
 \widehat{R}_X&=\bb  R_X &\overline{B}^H X\overline{A}B\\ B^H \overline{A}^H X\overline{B} & R+B^H(\overline{H}+\overline{A}^H X\overline{A}) B\eb.
\end{align}
Therefore, $\widehat{R}_X$ is nonsingular, i.e., $X\in \mathrm{dom}(\widehat{\mathcal{R}})$, since the Schur complement $R_X$ of $\widehat{R}_X$ given by
\begin{align*}
\widehat{R}_X/R_X&=R+B^H(\overline{H}+\overline{A}^H X\overline{A}) B-B^H(\overline{A}^H X\overline{B}R_X^{-1}\overline{B}^H X\overline{A}) B\\
&=R+B^H\overline{X}B=R_X
\end{align*}
is also nonsingular. This implies $\mathrm{dom}(\mathcal{R}) \subseteq \mathrm{dom}(\widehat{\mathcal{R}})$ and thus
the solution set $\mathcal{R}_= := \{X\in \mathrm{dom}(\mathcal{R})\, |\, X = \mathcal{R}(X) \}$ of the CDARE \eqref{cdare} is
contained in the solution set $\widehat{\mathcal{R}}_= := \{X\in \mathrm{dom}(\widehat{\mathcal{R}})\, |\, X = \widehat{\mathcal{R}}(X) \}$ of
the transformed DARE \eqref{dare2}.

Moreover, for the sake of simplicity the following sets and matrices
\begin{subequations}
\begin{align}
 &\mathcal{R}_\leq := \{X\in \mathrm{dom}(\mathcal{R})\, |\, X  \leq  \mathcal{R}(X) \},
 \mathcal{R}_\geq := \{X\in \mathrm{dom}(\mathcal{R})\, |\, X \geq \mathcal{R}(X) \},  \label{Rleq} \\
 &\widehat{\mathcal{R}}_\leq := \{X\in \mathrm{dom}(\widehat{\mathcal{R}})\, |\, X  \leq  \widehat{\mathcal{R}}(X) \},
 \widehat{\mathcal{R}}_\geq := \{X\in \mathrm{dom}(\widehat{\mathcal{R}})\, |\, X \geq \widehat{\mathcal{R}}(X) \}, \label{hRleq} \\
&F_X := R_X^{-1}B^H \overline{X}A, T_X := A-BF_X = (I + G \overline{X})^{-1}A, \widehat{T}_X  := \overline{T_X} T_X,  \label{Fx}  \\
&\widehat{F}_X := \widehat{R}_X^{-1}\widehat{B}^H X\widehat{A}, \widehat{T}_X^D := \widehat{A} - \widehat{B} \widehat{F}_X
= (I + \widehat{G} X)^{-1}\widehat{A}, D_{\widehat{F}} := \widehat{A} - \widehat{B} \widehat{F}   \label{hFx}
\end{align}
\end{subequations}
will be used throughout the paper for any $X\in \mathrm{dom}(\mathcal{R})$ and $\widehat{F}\in \mathbb{C}^{2m\times n}$.
}

The following lemma characterizes {some useful identities} depending on the matrix operators $\mathcal{R}(\cdot)$ and $\widehat{\mathcal{R}}(\cdot)$ with respect to the associated conjugate Stein and Stein operators,
which will be quoted in the proof of Lemma~\ref{thm3p3} later on.
\begin{Lemma}\cite{c.f21,f.c23} \label{lem2p1} \color{black}
\begin{subequations}\label{Req}
\begin{enumerate}
  \item[(i)] If $A_{F}:=
  A-BF$ for any $F\in\mathbb{C}^{m\times n}$ and $H_F:=H+F^H R F$, then
  \begin{align}
  X-\mathcal{R}(X)&=\mathcal{C}_{A_{F}}(X)-H_F+K_F(X), \label{Req-a}
  \end{align}
  where $K_F(X) := (F-F_{X})^H(R+B^H \overline{X} B)(F-F_{X})$.
  \item[(ii)] If $K(Y,X) := K_{F_{Y}}(X)$ and
  $H_{F_Y} :=H+K(Y,0)= H+F_{Y}^H R F_{Y}$,
  then \eqref{Req-a} can be rewritten as
  \begin{align}
  X-\mathcal{R}(X)&=\mathcal{C}_{T_{Y}}(X)-H_{F_Y}+K(Y,X). \label{Req-b}
  \end{align}
\item[(iii)] If $D_{\widehat{F}}:=
  \widehat{A}-\widehat{B}\widehat{F}$ for any $\widehat{F}\in\mathbb{C}^{2m\times n}$ and $\widehat{H}_F:=\widehat{H}+\widehat{F}^H \widehat{R} \widehat{F}$, then
  \begin{align}
  X-\widehat{\mathcal{R}}(X)&=
  \mathcal{S}_{D_{\widehat{F}}}(X)
  -\widehat{H}_{\widehat{F}}+\widehat{K}_{\widehat{F}}(X), \label{hReq-a}
  \end{align}
  where $\widehat{K}_{\widehat{F}}(X) := (\widehat{F}-\widehat{F}_{X})^H(\widehat{R}+\widehat{B}^H {X} \widehat{B})(\widehat{F}-\widehat{F}_{X})$ and $\widehat{F}_{X}:=\widehat{R}_X^{-1}\widehat{B}X\widehat{A}$.
\end{enumerate}
\end{subequations}
\end{Lemma}
%
Next, a sufficient condition is provided in the following result to guarantee the equivalence between two subsets of $\mathbb{H}_n$, which will be described in the proof of Lemma~\ref{thm3p3} later.
\begin{Lemma}\label{thm2p29}
Let $X\in\mathbb{H}_n$.
Suppose that there exists  a matrix $F\in\mathbb{C}^{m\times n}$ such that $\rho(\widehat{A}_F)<1$, where $\widehat{A}_F:=\overline{A}_FA_F$ and $A_F:=A-BF$.
Assume that there exists a  $n$-square matrix $Y\geq 0$ satisfies
\begin{align}\label{3}
\mathcal{C}_{T_X}(Y)\geq K_F(X).
\end{align}
Then,
$X\in\mathbb{T}$ if $X\in\mathbb{P}$.
\end{Lemma}
\begin{proof}
The inequality \eqref{3} implies
\begin{align}\label{4}
\mathcal{S}_{\widehat{T}_X}(Y)\geq K_F(X)+T_X^H \overline{K_F(X)} T_X.
\end{align}
Assume that there exists a scalar $\lambda$ with $|\lambda|\geq 1$ such that $\widehat{T}_X x=\lambda x$ with a nonzero vector $x\in\mathbb{C}^n$. Then,
\begin{align}\label{59}
0\geq(1-|\lambda|^2)(x^H Y x)=x^H \mathcal{S}_{\widehat{T}_X}(Y) x \geq x^H K_F(X) x+x^H T_X^H \overline{K_F(X)} T_X x \geq 0.
\end{align}
Let $W\in\mathbb{H}_n$ and $V\in\mathbb{P}$. Note that
\begin{subequations}
\begin{align}\label{6}
 \mathrm{Ker}(K_F(V))&\subseteq\mathrm{Ker}(A_F-T_V), \\
  \mathrm{Ker}(\overline{K_F(V)} T_W)&\subseteq\mathrm{Ker}(\overline{A_F}T_W-\overline{T_V}T_W).
\end{align}
\end{subequations}
 According to \eqref{59}  we have
\begin{align*}
  T_X x &= A_Fx,  \\
  \overline{T_X}T_X x&= \overline{A_F}T_X x.
\end{align*}
Thus,
\begin{align}\label{7}
\widehat{A}_{F}x=\overline{{A}_{F}}{A}_{F}x
=\overline{{A}_{F}}T_X x=\overline{T_X}T_X x=\widehat{T}_{X}x=\lambda x.
\end{align}
In other words, $\lambda\in\sigma(\widehat{A}_{F})$, and a contradiction is obtained from $\sigma(\widehat{A}_{F})\subseteq\mathbb{D}$.
\end{proof}
Let $\mathbb{T}:=\{X\in\mathrm{dom}(\mathcal{R}) \,|\, \rho(\widehat{T}_X)<1\}$ and $\mathbb{P}:=\{X\in \mathrm{dom}(\mathcal{R}) \,|\, R_X>0\}$.
Starting from an initial matrix $X_0 \in \mathcal{S}_{\geq} :=\bigcup\limits_{W \in\mathbb{T}} \{ X\in \mathbb{H}_n\, |\, \mathcal{C}_{T_W} (X)\geq H_W\}$ if $\mathbb{T}\neq\emptyset$.
It has been shown in \cite{f.c23} that the sequence $\{X_k\}_{k=0}^\infty$ generated by the following FPI of the form
\begin{equation} \label{fpi}
X_{k+1} = \mathcal{R}(X_k), \quad k\geq 0,
\end{equation}
converges at least linearly to the maximal solution of the CDARE \eqref{cdare} if
\begin{equation}\label{ma}
 \mathbb{T}\neq\emptyset\quad \mbox{and}\quad \mathcal{R}_{\leq}\cap\mathbb{P}\neq\emptyset,
\end{equation}
which are main assumptions throughout this paper. The following theorem is quoted from Theorem~3.1 of \cite{f.c23} and it will be applied to deduce the
inheritance properties presented in Section~\ref{sec:inheri}.
\begin{Theorem} \label{thm:Xmax}
{\color{black} If the assumptions in \eqref{ma} are fulfilled, the maximal solution $X_M$ of the CDARE \eqref{cdare} exists.}
Furthermore, the following statements hold:
\begin{enumerate}
  \item[(i)]
The sequence $\{X_k\}_{k=0}^\infty$ generated by the FPI \eqref{fpi} with $X_0\in\mathcal{S}_\geq$ is well-defined. Moreover,
   $X_k\in\mathcal{S}_\geq \subseteq \mathcal{R}_\geq \cap \mathbb{P} \cap \mathbb{T}$ for all $k\geq 0$.
  \item[(ii)] $X_k\geq X_{k+1}\geq X_\mathbb{P}$ for all $k\geq 0$ and $X_\mathbb{P}\in \mathcal{R}_{\leq} \cap \mathbb{P}$.
  \item[(iii)] The sequence $\{X_k\}_{k=0}^\infty$  converges at least linearly to $X_M$, which is the maximal element of the set $\mathcal{R}_\leq\cap \mathbb{P}$
and satisfies $\rho (\widehat{T}_{X_M}) \leq 1$, with the rate of convergence
  \[  \limsup_{k\rightarrow \infty} \sqrt[k]{\|X_k - X_{M}\|}\leq \rho (\widehat{T}_{X_M}) \]
 whenever $X_M\in\mathbb{T}$.
\end{enumerate}
\end{Theorem}

For any $Y_0 \in \mathrm{dom}(\mathcal{R})$, consider the sequence $\{Y_k\}_{k=0}^\infty$ generated by the FPI
\begin{equation} \label{fpi-hat}
  Y_{k+1} = \widehat{\mathcal{R}}(Y_k) := \mathcal{R}(\mathcal{R}(Y_k)),\quad k\geq 0,
\end{equation}
where the matrix operators $\mathcal{R}(\cdot)$ and $\widehat{\mathcal{R}}(\cdot)$ are defined by \eqref{cdare-a} and \eqref{dare2-a}, respectively.
Analogously, the existence of the maximal solution $Y_M$ to the transformed DARE \eqref{dare2} can be established by modifying the proof
Theorem~3.1 in \cite{f.c23} slightly. It is stated without proof in the following theorem, in which
$\widehat{\mathbb{T}} := \{X \in \mathbb{H}_n\,|\, \rho(\widehat{T}_X^D)<1\}$, $\widehat{\mathbb{P}} :=
\{Y\in \mathrm{dom}(\widehat{\mathcal{R}}) \,|\, \widehat{R}_Y > 0\}$ and
$\widehat{\mathcal{S}}_{\geq} :=\bigcup\limits_{X_{\widehat{\mathbb{T}}}\in\widehat{\mathbb{T}}} \{ X\in \mathbb{H}_n\, |\, \mathcal{S}_{X_{\widehat{\mathbb{T}}}} (X)\geq \widehat{H}_{X_\mathbb{T}}=\widehat{H}
 +\widehat{F}_{X_\mathbb{T}}^H \widehat{R} \widehat{F}_{X_\mathbb{T}}\}$ if $\widehat{\mathbb{T}}\neq\emptyset$.

\begin{Theorem} \label{thm:Xmax-dare}
Assume that $\widehat{\mathbb{T}}\neq \emptyset$ and $\widehat{\mathcal{R}}_{\leq}\cap\widehat{\mathbb{P}}\neq\emptyset$. Let the sequence $\{Y_k\}_{k=0}^\infty$ be generated by the FPI \eqref{fpi-hat} with $Y_0\in\widehat{\mathcal{S}}_\geq$. Then the following statements hold:
\begin{enumerate}
  \item[(i)] ${Y_k}$ is well-defined and $Y_k\in\widehat{\mathcal{S}}_\geq\subset\widehat{\mathcal{R}}_\geq\cap\widehat{\mathbb{P}}\cap\widehat{\mathbb{T}}$ for all $k\geq 0$.
  \item[(ii)] $Y_k\geq Y_{k+1}\geq X_{\widehat{\mathbb{P}}}$ for all $X_{\widehat{\mathbb{P}}}\in\widehat{\mathbb{P}}$ and for all $k\geq 0$.
  \item[(iii)] The sequence $\{Y_k\}_{k=0}^\infty$  converges at least linearly to $Y_M$, which is the maximal element of $\widehat{\mathcal{R}}_\leq\cap
      \widehat{\mathbb{P}}$ satisfying $\rho (\widehat{T}^D_{Y_M}) \leq 1$, with the rate of convergence
  \[  \limsup_{k\rightarrow \infty} \sqrt[k]{\|Y_k - Y_{M}\|}\leq \rho (\widehat{T}^D_{Y_M}),
  \]
 provided that $Y_M\in\widehat{\mathbb{T}}$.
\end{enumerate}
\end{Theorem}
It is clear that $X_M\leq Y_M$, and the interest is finding some reasonable conditions such that $X_M$ coincides with $Y_M$.
Finally, the following lemma can be used to prove the inheritance of the maximal solution to CDAREs,
which will be illustrated in the next section.
\begin{Lemma}\label{thm3p3}
Assume that $\mathcal{R}_{\leq}\cap\mathbb{P}\neq\phi$. Let $\mathbb{F}:=\{F\in\mathbb{C}^{m\times n}|\rho(\widehat{A}_F)<1\}$ and  $\widehat{\mathbb{F}}:=\{\widehat{F}\in\mathbb{C}^{2m\times n}|\rho(\widehat{A}-\widehat{B}\widehat{F})<1\}$.
The following two statements are equivalent:
\begin{enumerate}
  \item[(i)] $\mathbb{F}\neq\phi$.
  %
  \item[(ii)] $\mathbb{T}\cap\mathbb{P}\neq\phi$. 
\end{enumerate}
Moreover, if either (i) or (ii) holds, then $\widehat{\mathcal{R}}_{=}\cap\widehat{\mathbb{P}}\neq\phi$ and
the following two statements are also holded:
\begin{enumerate}
\item[(iii)]$\widehat{\mathbb{F}}\neq\phi$ or the pair $(\widehat{A},\widehat{B})$ is stabilizable.
 \item[(iv)]
 $\widehat{\mathbb{T}}\cap\widehat{\mathbb{P}}\neq\phi$.

\end{enumerate}
\end{Lemma}
\begin{proof}
The proof of first two equivalence of these statements is given as follows.
\begin{enumerate}
  \item[(i)$\Rightarrow$ (ii):]

Let $X_\star = \mathcal{C}_{A_F}^{-1} (Z)$ by taking an arbitrary matrix $Z$ with $Z\geq H_F$  and $X_{\mathbb{P}}\in\mathcal{R}_{\leq}\cap\mathbb{P}$.
Recall that~\eqref{Req-a} we have
 \begin{align*}
 X_{\mathbb{P}}-\mathcal{R}(X_{\mathbb{P}})&=\mathcal{C}_{A_{F}}(X_{\mathbb{P}})-H_F+K_F(X_{\mathbb{P}})\leq 0.
 \end{align*}
Therefore, $\mathcal{C}_{A_{F}}(X_{\mathbb{P}})\leq H_F\leq Z\leq\mathcal{C}_{A_{F}}(X_{\star})$. Namely, $X_{\star}\geq X_{\mathbb{P}}$ and $R+B^H X_{\star} B \geq R+B^H X_{\mathbb{P}} B>0$. Thus, $X_{\star}\in\mathbb{P}\subseteq\mbox{dom}({\mathcal{R}})$.
On the other hand, we also have
\begin{align*}
X_\star-\mathcal{R}(X_\star)&=\mathcal{C}_{A_{F}}(X_\star)-H_F+K_F(X_\star)\geq K_F(X_\star).
 \end{align*}
Applying \eqref{Req-b} we obtain
\begin{align*}
&X_\star-\mathcal{R}(X_\star) =\mathcal{C}_{T_{X_\star}}(X_\star)-H_{F_{X_\star}},\\
X_{\mathbb{P}}-\mathcal{R}(X_{\mathbb{P}}) &=\mathcal{C}_{T_{X_\star}}(X_{\mathbb{P}})-H_{F_{X_\star}}+K(X_\star,X_{\mathbb{P}}).
 \end{align*}
From which we deduce that
\begin{subequations}
\begin{align}
  &\mathcal{C}_{T_{X_\star}}(X_\star-X_{\mathbb{P}})
=X_\star-\mathcal{R}(X_\star)+H_{F_{X_\star}}
  -(X_{\mathbb{P}}-\mathcal{R}(X_{\mathbb{P}}) +H_{F_{X_\star}}\label{3a}\\
  &-K(X_\star,X_{\mathbb{P}}))\geq K(X_\star,X_{\mathbb{P}})-(X_{\mathbb{P}}-\mathcal{R}(X_{\mathbb{P}}))+K_F(X_\star) \geq K_F(X_\star).\label{3b} 
\end{align}
\end{subequations}
It follows from Lemma~\ref{thm2p29} with the assumption $\rho (\widehat{A}_F) < 1$ that $\rho (\widehat{T}_{X_\star}) < 1$. That is, $X_\star\in\mathbb{T}\cap\mathbb{P}$.
  \item[(ii)$\Rightarrow$(i):] Suppose that the statement (ii) holds. If we let $F:=F_{X_\star}$, then
  $\rho (\widehat{A}_F) = \rho (\widehat{T}_{X_\star}) < 1$.
  \end{enumerate}
  Suppose that $\mathcal{R}_{\leq}\cap\mathbb{P}\neq\phi$ and (ii) holds. From Theorem~\ref{thm:Xmax} the maximal solution $X_M$ exists and thus
  $\widehat{\mathcal{R}}_{\leq}\cap\widehat{\mathbb{P}}\neq\phi$.
  The proof of the remaining part of the lemma are listed below.
   \begin{enumerate}
   \item[(i)$\Rightarrow$(iii):]Let $\widehat{F}:=\bb \overline{F}A_F \\ F-F_H\eb\in\mathbb{C}^{2m\times n}$, then $\widehat{A}-\widehat{B}\widehat{F} = \widehat{A}_F $ and this complete the proof of (iii).
   \item[(iii)$\Rightarrow$(iv):]
   Let $X_\star = \mathcal{S}_{D_{\widehat{F}}}^{-1} (Z)$ by taking an arbitrary matrix $Z$ with $Z\geq \widehat{H}_{\widehat{F}}$  and $X_{\widehat{\mathbb{P}}}\in\widehat{\mathcal{R}}_{\leq}\cap\widehat{\mathbb{P}}$.
Recall that~\eqref{hReq-a} we have
 \begin{align*}
 X_{\widehat{\mathbb{P}}}-\widehat{\mathcal{R}}(X_{\widehat{\mathbb{P}}})&=\mathcal{S}_{D_{\widehat{F}}}(X_{\widehat{\mathbb{P}}})
 -\widehat{H}_{\widehat{F}}+\widehat{K}_{\widehat{F}}(X_{\widehat{\mathbb{P}}})\leq 0.
 \end{align*}
Therefore, $\mathcal{S}_{D_{\widehat{F}}}(X_{\widehat{\mathbb{P}}})\leq \widehat{H}_{\widehat{F}}\leq\mathcal{S}_{D_{\widehat{F}}}(X_{\star})$. Namely, $X_{\star}\geq X_{\widehat{\mathbb{P}}}$ and $\widehat{R}+\widehat{B}^H X_{\star} \widehat{B} \geq \widehat{R}+\widehat{B}^H X_{\widehat{\mathbb{P}}} \widehat{B}>0$. Thus, $X_{\star}\in\widehat{\mathbb{P}}\subseteq\mbox{dom}({{\widehat{R}}})$.
On the other hand, we also have
\begin{align*}
X_\star-\widehat{\mathcal{R}}(X_\star)&=\mathcal{S}_{D_{\widehat{F}}}
(X_\star)-\widehat{H}_{\widehat{F}}+\widehat{K}_{\widehat{F}}(X_\star)\geq \widehat{K}_{\widehat{F}}(X_\star).
 \end{align*}
Applying \eqref{Req-b} we obtain
\begin{align*}
&X_\star-\widehat{\mathcal{R}}(X_\star) =\mathcal{S}_{T^D_{X_\star}}(X_\star)-\widehat{H}_{\widehat{F}_{X_\star}},\\
X_{\widehat{\mathbb{P}}}-\widehat{\mathcal{R}}(X_{\widehat{\mathbb{P}}}) &=\mathcal{S}_{T^D_{X_\star}}(X_{\widehat{\mathbb{P}}})-\widehat{H}_{\widehat{F}_{X_\star}}+\widehat{K}(X_\star,X_{\widehat{\mathbb{P}}}).
 \end{align*}
From which we deduce that
\begin{align*}
  &\mathcal{S}_{T^D_{X_\star}}(X_\star-X_{\widehat{\mathbb{P}}})
=X_\star-\widehat{\mathcal{R}}(X_\star)+\widehat{H}_{\widehat{F}_{X_\star}}
  -(X_{\widehat{\mathbb{P}}}-\widehat{\mathcal{R}}(X_{\widehat{\mathbb{P}}}) +\widehat{H}_{\widehat{F}_{X_\star}}\\
  &-\widehat{K}(X_\star,X_{\widehat{\mathbb{P}}}))\geq \widehat{K}(X_\star,X_{\widehat{\mathbb{P}}})-(X_{\widehat{\mathbb{P}}}-\widehat{\mathcal{R}}(X_{\widehat{\mathbb{P}}}))+\widehat{K}_{\widehat{F}}(X_\star) \geq \widehat{K}_{\widehat{F}}(X_\star).
\end{align*}
It follows from Lemma~\ref{thm2p29} with the assumption $\rho ({D}_{\widehat{F}}) < 1$ that $\rho (\widehat{T}^D_{X_\star}) < 1$. That is, $X_\star\in\widehat{\mathbb{T}}\cap\widehat{\mathbb{P}}$.
\item[(iv)$\Rightarrow$(iii):] Suppose that the statement (iv) holds. If we let $\widehat{F}:=\widehat{F}_{X_\star}$, then
  $\rho (D_{\widehat{F}}) = \rho (\widehat{T}^D_{X_\star}) < 1$.

   \end{enumerate}
\end{proof}
\begin{Remark}\label{r1}
\par\noindent
\begin{enumerate}
  \item Assume that $\mathcal{R}_{\leq}\cap\mathbb{P}\neq\phi$. Recall that $\phi\neq\mathcal{S}_\geq\subset\mathcal{R}_\geq\cap\mathbb{P}\cap\mathbb{T}
    \subset\mathbb{T}\cap\mathbb{P}$ if $\mathbb{T}\neq\phi$. Thus, $\mathbb{T}\neq\phi$ if and only if $\mathbb{T}\cap\mathbb{P}\neq\phi$ if and only if $\mathbb{F}\neq\phi$, an analogous procedure yields
   $\widehat{\mathbb{T}}\neq\phi$ if and only if $\widehat{\mathbb{F}}\neq\phi$.
  \item
  We conclude that
    $\mathcal{S}_{\geq}\equiv\mathcal{S}_{\geq}^F:= \bigcup\limits_{F\in\mathbb{F}} \{ X\in \mathbb{H}_n\, |\, \mathcal{C}_{A_F} (X)\geq H_{F}=H+F^HRF\}$ if $\mathbb{T}\neq\phi$  and
   $\widehat{\mathcal{S}}_{\geq}\equiv\widehat{\mathcal{S}}_{\geq}^F:= \bigcup\limits_{\widehat{F}\in\widehat{\mathbb{F}}} \{ X\in \mathbb{H}_n\, |\, \mathcal{S}_{\widehat{A}-\widehat{B}\widehat{F}} (X)\geq \widehat{H}_{\widehat{F}}=\widehat{H}+\widehat{F}^H\widehat{R}\widehat{F}\}$ if $\widehat{\mathbb{T}}\neq\phi$.
   Indeed, $X\in\mathcal{S}_{\geq}$ implies that $\mathcal{C}_{T_X}(X)\geq H_{F_X}$ for
   some $X\in\mathbb{T}$ with $\rho(\widehat{T}_X)<1$. Then, $\mathcal{C}_{A_{F_X}}(X)\geq H_{F_X}$ since $T_X\equiv A-BF_X$. Thus, $\mathcal{S}_{\geq}\subset\mathcal{S}_{\geq}^F$. Conversely, let $\mathbb{F}:=\{F\in\mathbb{C}^{m\times n}|\rho(\widehat{A}_F)<1\}$ and $X_\star = \mathcal{C}_{A_F}^{-1} (H_F)$ for an element $F\in \mathbb{F}$. Then, $\mathcal{C}_{A_F}(X_\star)=H_F$, i.e., $X_\star\in\mathcal{S}_{\geq}^F$. The same arguments relying on Theorem~\ref{thm3p3} provide that $X_\star\in\mathbb{T}$. 
Furthermore,
\begin{align*}
\mathcal{C}_{T_{X_\star}}(X_\star)-H_{F_{X_\star}}&=X_\star-\mathcal{R}(X_\star)=\mathcal{C}_{A_F}(X_\star)-H_F
+K_F(X_\star)\\
&=K_F(X_\star)\geq 0.
\end{align*}
That is, $X_\star\in\mathcal{S}_{\geq}$. We conclude that $\mathcal{S}_{\geq}^F\subset\mathcal{S}_{\geq}$. Therefore, $\mathcal{S}_{\geq}=\mathcal{S}_{\geq}^F$.

 A similar proof can be given for the
$\widehat{\mathcal{S}}_{\geq}=\widehat{\mathcal{S}}_{\geq}^F$.
\end{enumerate}

\end{Remark}

\section{Inheritance properties of the CDARE} \label{sec:inheri}

In this section, as mentioned previously, we shall clarify which properties will be inherited from the maximal solution of the CDARE \eqref{cdare}.
Some results of the previous section will be applied in the following properties.
\begin{Proposition} \label{lem:inheri}
Assume that $\mathcal{R}_\leq \cap \mathbb{P} \neq \emptyset$. The following statements hold:
\begin{enumerate}
  \item[(i)] 
  $\mathbb{T}\subseteq\widehat{\mathbb{T}}$ and $\widehat{\mathbb{F}}\neq\phi$ if ${\mathbb{F}\neq\phi}$.
  \item[(ii)] 
  $\mathbb{P}\subseteq\widehat{\mathbb{P}}$. The converse is also true if $X\in\mathcal{R}_=\cap \mathbb{P}$.
  \item[(iii)] $\mathcal{S}_\geq=\mathcal{S}_\geq^F\subseteq\widehat{\mathcal{S}}_{\geq}^F=\widehat{\mathcal{S}}_\geq$.
\end{enumerate}
\end{Proposition}
\begin{proof}
\par\noindent
\begin{enumerate}
\item[(i)]
The result is a direct consequence of a part of  Lemma~\ref{thm3p3} and Remark~\ref{r1}.
\item[(ii)]
  Observed that
\begin{align*}
 \widehat{R}_X&=\bb  R_X &\overline{B}^H X\overline{A}B\\ B^H \overline{A}^H X\overline{B} & R+B^H(\overline{H}+\overline{A}^H X\overline{A}) B\eb.
 \end{align*}
Thus, $R_X>0$ if  $\widehat{R}_X>0$.

 Conversely, the positiveness of $\widehat{R}_X$ can be verified directly by going through the
 following computation
 \begin{align*}
 \widehat{R}_X &=
 \bb  R_X &\overline{B}^H X\overline{A}B\\ B^H \overline{A}^H X\overline{B} & R+B^H(\overline{X}+\overline{A}^H X\overline{B}(\overline{R_X})^{-1}\overline{B}^H
 X\overline{A}) B\eb\\
 &=0\oplus R_X+\bb \overline{R_X} \\ B^H\overline{A}^HX\overline{B}  \eb
 (\overline{R_X})^{-1}\bb  \overline{R_X} &
 \overline{B}^H X \overline{A} B
  \eb>0,
 \end{align*}
 if $X\in\mathcal{R}_=$ and $R_X>0$.
 \item[(iii)]
   Let $X\in{\mathcal{S}}_{\geq}={\mathcal{S}}_{\geq}^F$. Namely, $\mathcal{C}_{A_F} (X)\geq H_{F}=H+F^HRF$ for some $F\in\mathbb{C}^{m\times m}$ with $\rho(\widehat{A}_F)<1$, or
 \begin{align}\label{123}
 \mathcal{S}_{\widehat{A}_F}(X)\geq H_{F}+A_F ^H\overline{H_{F}}A_F .
 \end{align}
Following the formula~\eqref{Req-a} we have
\begin{align*}
  &K_F(H)-K_F(0)=(H-R(H)-\mathcal{C}_{A_F}(H)+H_F)-(0-R(0)-\mathcal{C}_{A_F}(0)\\  &+H_F)=A_F^H\overline{H}A_F-A^H\overline{H}(I+G\overline{H})^{-1}A.
\end{align*}
Hence,
  \begin{align*}
&H_{F}+A_F ^H\overline{H}_{F}A_F=H+K_F(0)+A_F ^H\overline{H}A_F+A_F ^H\overline{K_F(0)}A_F \\
&=H+K_F(0)+ K_F(H)-K_F(0)+A^H\overline{H}(I+G\overline{H})^{-1}A+A_F ^H\overline{K_F(0)}A_F \\
&=H+A^H\overline{H}(I+G\overline{H})^{-1}A+A_F ^H\overline{K_F(0)}A_F + K_F(H)\\
&=\widehat{H}+\widehat{F}^H\widehat{R}\widehat{F}=\widehat{H}_{\widehat{F}},
 \end{align*}
 where $\widehat{F}:=\bb \overline{F}A_F \\ F-F_H\eb\in\mathbb{C}^{2m\times n}$.
Note that
\[
\widehat{F}^H\widehat{R}\widehat{F}=A_F ^H\overline{K_F(0)}A_F + K_F(H).
\]
  It is also not difficult to see that
 \begin{align}\label{456}
 \widehat{A}_F=(\overline{A}-\overline{BF})(A-BF)=D_{\widehat{F}}:=\widehat{A}-\widehat{B}\widehat{F}.
 \end{align}
Therefore, \eqref{123} is equivalent to
\[
\mathcal{S}_{D_{\widehat{F}}}(X)=\mathcal{S}_{\widehat{A}_F} (X)\geq H_{F}+A_F ^H\overline{H}_{F}A_F= \widehat{H}_{\widehat{F}}.
\]
That is, $X\in\widehat{\mathcal{S}}_{\geq}^F=\widehat{\mathcal{S}}_{\geq}$.
\end{enumerate}
\end{proof}

Under some reasonable conditions, it will be shown that the transformed DARE \eqref{dare2} inherits the maximal solution $X_M$ of the CDARE \eqref{cdare}
in the following theorem.
\begin{Theorem} \label{thm:inheri-1}
 If the assumptions in \eqref{ma} and \eqref{a1} hold, then the CDARE \eqref{cdare} and its associated DARE \eqref{dare2}
 have the same maximal solution $X_M$.
\end{Theorem}
\begin{proof}
From Theorem~\ref{thm:Xmax}, we have deduced that the CDARE \eqref{cdare} has the maximal solution $X_M\in \mathbb{P}$
with $\rho (\widehat{T}_{X_M})  \leq 1$ under the assumptions in \eqref{ma}. That is, $X_M \in \mathcal{R}_= \cap \mathbb{P}$ and hence
$X_M \in \widehat{\mathcal{R}}_= \cap \widehat{\mathbb{P}}$ and $\widehat{\mathbb{F}} \neq \emptyset$ follow from Proposition~\ref{lem:inheri}.
The result together with Proposition~\ref{lem:inheri} imply that the sufficient conditions of Theorem~\ref{thm:Xmax-dare} are fulfilled.

If we select $Y_0 = X_0 \in \mathcal{S}_\geq$, then $Y_0 \in \widehat{\mathcal{S}}_\geq$ is also true from the part (iii) of Proposition~\ref{lem:inheri}.
Consequently, from Theorem~\ref{thm:Xmax-dare}, the sequence $\{X_k\}_{k=0}^\infty$ generated by the FPI \eqref{fpi} converges nonincreasingly to the maximal
solution $X_M$ of the CDARE \eqref{cdare}.

On the other hand, from Theorem~\ref{thm:Xmax-dare} with $Y_0=X_0$, the sequence $\{Y_k\}_{k=0}^\infty$ generated by the FPI \eqref{fpi-hat}
converges nonincreasingly to the maximal solution $Y_M$ of its induced DARE \eqref{dare2}. Moreover, it is easily seen that
$Y_k  = X_{2k}\in \mathrm{dom}(\mathcal{R})$ for $k\geq 1$. That is, $\{Y_k\}_{k=0}^\infty$ is a subsequence of $\{X_k\}_{k=0}^\infty$
and thus they must have the same limit $X_M = Y_M$.
\end{proof}
In addition, according to Theorem~\ref{thm:inheri-1}, it is natural to determine whether the closed-loop matrices $\widehat{T}_{X_M}$ and
$\widehat{T}_{X_M}^D$ defined by \eqref{Fx} and \eqref{hFx}, respectively, might have the same structure. Indeed, the invariance of these closed-loop matrices
follows immediately from a more general result presented below.

\begin{Theorem} \label{thm:inheri-2}
  Assume that $X\in\mathrm{dom}(\mathcal{R})$. Then,
  \begin{enumerate}
    \item $\mathcal{R}(X)\in\mathrm{dom}(\mathcal{R})$.
    \item 
        $\widehat{T}^D_X=\overline{T_X}T_{\mathcal{R}(X)}$.
In particular, the closed-loop matrices $\widehat{T}_X$ and $\widehat{T}_X^D$ coincide for any $X\in\mathcal{R}_=$.
  \end{enumerate}
\end{Theorem}
\begin{proof}
Applying the Sherman-Morrison-Woodbury formula (SMWF) with \eqref{hatG}, we see that two matrices
\begin{align*}
\Omega_X &:=(I+G\overline{\mathcal{R}(X)})^{-1}=(I+{{G}}\overline{H}+G\overline{A}^HX \Theta_X\overline{A})^{-1}\\
&= \overline{\Theta_H} - \overline{\Theta_H} G\overline{A}^HX (I+\widehat{G}X)^{-1} \overline{A} \overline{\Theta_H}\\
(I+\widehat{G}X)^{-1} &
=\Theta_X- \Theta_X \overline{A}\Omega_X G\overline{A}^HX\Theta_X.
\end{align*}
exist, where $\Theta_X  := (I+\overline{{G}}X)^{-1}$.
Therefore, we conclude that
\begin{align*}
  \widehat{T}_X^D &= (I+\widehat{G}X)^{-1}\widehat{A}=\overline{T_X}\Omega_X
  (I+G(\overline{\mathcal{R}(X)}-\overline{A}^HX\Theta_X\overline{A}))(I+G\overline{H})^{-1}A\\
&=\overline{T_X}\Omega_X
  (I+G\overline{H})(I+G\overline{H})^{-1}A=\overline{T_X}T_{\mathcal{R}(X)} = \widehat{T}_X^C.
\end{align*}
Especially, if $X\in \mathcal{R}_=$, then $X\in \widehat{\mathcal{R}}_=$ and thus $\widehat{T}_X^D = \widehat{T}_X$.
\end{proof}
\section{An accelerated fixed-point iteration} \label{sec:afpi}

According to Theorem~\ref{thm:inheri-1}, it is natual to design numerical methods for computing the maximal solution $X_M$ of
the CDARE \eqref{cdare} via its tansformed DARE \eqref{dare2-b} under the assumptions in \eqref{ma} and $H\in \mathrm{dom}(\mathcal{R})$.
Although there are dozens of numerical methods presented in the existing literatures for solving standard DAREs,
we are mainly concerned with the possibility whether the FPI \eqref{fpi-hat} can be accelerated at least superlinearly for finding the maximal solution
of the CDARE \eqref{cdare} in this section.

Since the FPI is usually linearly convergent, a numerical method of higher order of convergence is always required in the practical computation
and many real-life applications. Recently, we have proposed an accelerated fixed-point iteration \cite{c.f21} for computing the extremal solutions of a standard DARE
\begin{equation} \label{dare2-Y}
Y = \mathcal{R}_d (Y) := \widehat{A}^H Y (I + \widehat{G} Y)^{-1}\widehat{A} + \widehat{H},
\end{equation}
ith $\widehat{G}\geq 0$ and $\widehat{H}\geq 0$. Therefore, we shall extend the idea of AFPI for solving
the transformed DARE \eqref{dare2-b} with $\widehat{G}, \widehat{H}\in \mathbb{H}_n$ and provide a corresponding convergence theorem in this section.

For any positive integer $r > 1$, we will revisit an accelerated fixed-point  iteration of the form
\begin{subequations} \label{afpi-Yk}
\begin{align}
  \widehat{Y}_{k+1} &= \mathcal{R}_d^{(r^{k+1} - r^{k})}(\widehat{Y}_k),\quad k\geq 1,\\
  \widehat{Y}_{1} &=\mathcal{R}_d^{(r)}(\widehat{Y}_0),\quad k=1
\end{align}
\end{subequations}
 with $\widehat{Y}_{0}=Y_0$, for computing the numerical solutions of DARE \eqref{dare2-Y}.
Here $\mathcal{R}_d^{(\ell)}(\cdot)$ denotes the composition of the operator $\mathcal{R}_d (\cdot)$ itself for $\ell$ times, where $\ell \geq 1$ is a positive integer.
Theoretically, the iteration of the form \eqref{afpi-Yk} is equivalent to the formula
\begin{align} \label{afpi-eqform}
    \widehat{Y}_{k} &= \mathcal{R}_d^{(r^k)}(\widehat{Y}_0),\quad k\geq 1,
\end{align}
with $\widehat{Y}_0 = Y_0$, and we see that $\widehat{Y}_k = Y_{r^{k}}$ for each $k\geq 1$.

For the FPI defined by $Y_{k+1} = \mathcal{R}_d (Y_k)$, in which $Y_0\in \mathbb{H}_n$ is an initial matrix and
the operator $\mathcal{R}_d (\cdot)$ is defined by \eqref{dare2-Y},
it is shown in \cite{l.c18} that the fixed-point iteration can be rewritten as the following formulation
 \begin{equation}\label{fpi-eqform}
 Y_{k+1} = \mathcal{R}_d^{(k)}(\mathcal{R}_d (Y_0))= \mathcal{R}_d^{(k+1)}(Y_0) = H_{k}+A_{k}^H Y_0 (I+ G_{k} Y_0)^{-1} A_{k},
 \end{equation}
where the sequence of matrices $\{ (A_k, G_k, H_k)\}_{k=0}^\infty$ is generated by the following iteration
\begin{equation} \label{opF}
\mathbb{X}_{k+1} = F(\mathbb{X}_k, \mathbb{X}_0) :=
\begin{bmatrix} A_0\Delta_{G_{k},{H_0}}A_{k}\\G_0+A_0\Delta_{G_{k},{H_0}}G_{k}A_0^H\\H_{k}+A_{k}^H H_0\Delta_{G_{k},{H_0}}A_{k}\end{bmatrix},
\end{equation}
with $\mathbb{X}_k := \left[ A_k^H\  G_k \ H_k \right]^H$ and $\mathbb{X}_0:= \left[ \widehat{A}^H\  \widehat{G}\ \widehat{H} \right]^H$ for each $k\geq 0$,
provided that the matrices $\Delta_{G_i, H_j} := (I+G_i H_j)^{-1}$ exists for all $i,j\geq 0$.
Notice that $F: \mathbb{K}_n \times \mathbb{K}_n \rightarrow \mathbb{K}_n$ is a binary operator with
$\mathbb{K}_n:= \mathbb{C}^{n\times n}\times \mathbb{H}_n \times \mathbb{H}_n$.
Moreover, it has been shown in Theorem 4.2 of \cite{l.c20} that the iteration \eqref{opF} has the semigroup property and thus satisfies
the so-called discrete flow property~\cite{l.c20}, that is,
\begin{align}\label{DF}
{\mathbb{X}}_{i+j+1}=F({\mathbb{X}}_{i},{\mathbb{X}}_{j})
\end{align}
for any nonnegative integers $i$ and $j$. Here the subscript of \eqref{DF} is an equivalent adjustment to the original formula presented in \cite[Theorem 3.2]{l.c20}.

Analogously, in order to characterize the construction of the operator $\mathcal{R}_d^{(r^k)}(\cdot )$ defined by \eqref{afpi-eqform} with $r > 1$,
we define the operator $\mathbf{F}_\ell: \mathbb{K}_n \rightarrow \mathbb{K}_n$ iteratively by
\begin{equation} \label{bfF}
   \mathbf{F}_{\ell+1} (\mathbb{X}) = F(\mathbb{X}, \mathbf{F}_\ell (\mathbb{X})),\quad \ell\geq 1,
\end{equation}
with $\mathbf{F}_1 (\mathbb{X}) = \mathbb{X}$ for all $\mathbb{X}\in \mathbb{K}_n$ and$F(\cdot, \cdot)$ being defined by \eqref{opF}.
Furthermore, if we let $\mathbf{X}_k := \left[\mathbf{A}_k^H\ \mathbf{G}_k\ \mathbf{H}_k \right]^H\in \mathbb{K}_n$ for $k\geq 0$ and
$\mathbf{X}_0 := \mathbb{X}_0$ be defined as in \eqref{opF}, then the sequence $\{\mathbf{A}_k, \mathbf{G}_k, \mathbf{H}_k \}$ generated by the iteration
\begin{equation} \label{afpi-r2}
   \mathbf{X}_{k+1} = F(\mathbf{X}_k, \mathbf{X}_k) = \mathbf{F}_2 (\mathbf{X}_k),\quad k\geq 0,
\end{equation}
which is equivalent to the doubling or strctured doubling algorithms \cite{a77,c.f.l.w04}.
Indeed, according the semigroup and discrete flow property~\eqref{DF},
we have $\mathbf{A}_k = A_{2^{k}-1}$, $\mathbf{G}_k = G_{2^{k}-1}$ and $\mathbf{H}_k = H_{2^{k}-1}$ for each $k\geq 0$.
That is, under the iteration \eqref{afpi-r2}, the sequence of matrices $\{(A_k, G_k, H_k)\}_{k=0}^\infty$ proceeds rapidly with their subscripts being the
exponential numbers of base number $r=2$.From the theoretical point of view, staring with a suitable matrix $Y_0\in \mathbb{H}_n$, the sequence $\{Y_k\}_{k=0}^\infty$ generated by \eqref{fpi-eqform} might also converge rapidly to its limit, if the limit exists.

Recently, for any positive integer $r>1$, Lin and Chiang proposed an efficient iterative method for generating
the sequence $\{(A_k,G_k,H_k)\}_{k=0}^\infty$ with order $r$ of R-convergence, provided that the operator $F(\cdot, \cdot)$ in \eqref{opF}
is well-defined, see, e.g., Algorithm 3.1 in \cite{l.c20}. Theoretically, this algorithm utilizes the following accelerated iteration
\begin{equation}\label{afpi-r}
    \mathbf{X}_{k+1} = \mathbf{F}_r (\mathbf{X}_k),\quad k\geq 0,
\end{equation}
with $\mathbf{X}_0:= \left[\widehat{A}^H\ \widehat{G}\ \widehat{H} \right]^H$ and $\mathbf{F}_r (\cdot)$ being the operator defined by \eqref{bfF},
for constructing $\mathbf{A}_k = A_{r^{k}-1}$, $\mathbf{G}_k = G_{r^{k}-1}$ and $\mathbf{H}_k = H_{r^{k}-1}$, respectively.
Therefore, combining \eqref{fpi-eqform} and \eqref{afpi-r}, we obtain the pseudocode of AFPI summarized in Algorithm~\ref{alg-afpi}.
\begin{algorithm}
 \caption{The Accelerated Fixed-Point Iteration with $r$ (AFPI($r$)) for solving the DARE \eqref{dare2-Y}.
}
\label{alg-afpi}
\begin{algorithmic}
\REQUIRE $\mathbf{A}_0=\widehat{A}\in \mathbb{C}^{n\times n}$, $\mathbf{G}_0 = \widehat{G}\in \mathbb{H}_n$, $\mathbf{H}_0=\widehat{H}\in \mathbb{H}_n$,
$\widehat{Y}_0\in \mathbb{H}_n$ and $r > 1$.
\ENSURE the maximal solution $\widehat{Y}_{k_1}$ to the DARE \eqref{dare2-Y}.

\FOR{$k =0,1,2,\ldots$}
  \STATE $A_k^{(1)} = \mathbf{A}_k$,  $G_k^{(1)} = \mathbf{G}_k$,  $H_k^{(1)} = \mathbf{H}_k$;
  \FOR{$l = 1,2,\ldots,  r-2$}
    \STATE $A_k^{(l+1)} = A_k^{(l)} (I + \mathbf{G}_k H_k^{(l)})^{-1}\mathbf{A}_k$;
    \STATE $G_k^{(l+1)} = G_k^{(l)} + A_k^{(l)} (I + \mathbf{G}_k H_k^{(l)})^{-1}\mathbf{G}_k (A_k^{(l)})^H$;
    \STATE $H_k^{(l+1)} = \mathbf{H}_k + \mathbf{A}_k^H H_k^{(l)}  (I + \mathbf{G}_k H_k^{(l)})^{-1}\mathbf{A}_k$;
  \ENDFOR
  \STATE $\mathbf{A}_{k+1} = A_k^{(r-1)} (I + \mathbf{G}_k H_k^{(r-1)})^{-1}\mathbf{A}_k$;
  \STATE $\mathbf{G}_{k+1} = G_k^{(r-1)} + A_k^{(r-1)} (I + \mathbf{G}_k H_k^{(r-1)})^{-1}\mathbf{G}_k (A_k^{(r-1)})^H$;
  \STATE $\mathbf{H}_{k+1} = \mathbf{H}_k + \mathbf{A}_k^H H_k^{(r-1)} (I + \mathbf{G}_k H_k^{(r-1)})^{-1}\mathbf{A}_k$;
  \STATE $\widehat{Y}_{k+1} = \mathbf{A}_{k+1}^H \widehat{Y}_0 (I + \mathbf{G}_{k+1} \widehat{Y}_0)^{-1}\mathbf{A}_{k+1} + \mathbf{H}_{k+1}$;
  \IF{$\widehat{Y}_{k_1}$ satisfies the stopping criterion for some positive integer $k_1$}
    \STATE \RETURN $\widehat{Y}_{k_1}$;
  \ENDIF
\ENDFOR
\end{algorithmic}
\end{algorithm}

From Theorem~\ref{thm:inheri-1}, note that $X_M$ is a Hermitian solution of the transformed DARE \eqref{dare2-b} under the assumptions in \eqref{ma} and
$H\in \mathrm{dom}(\mathcal{R})$, in which the coefficient maxtices $\widehat{A}$, $\widehat{G}$ and $\widehat{H}$ are given by \eqref{coeff-dare2}.
Consequently, the maximal solution $X_M$ of the CDARE \eqref{cdare} can be computed numerically
by Algorithm~\ref{alg-afpi} with $\mathbf{A}_0 = \widehat{A}$, $\mathbf{G}_0 = \widehat{G}$ and $\mathbf{H}_0 = \widehat{H}$, respectively, and
the superlinear convergence of the AFPI performed by Algorithm~\ref{alg-afpi} will be stated as follows.

\begin{Theorem} \label{thm:afpi-conv}
Assume that the hypotheses of Theorem~\ref{thm:inheri-1} hold and let $\widehat{A}$, $\widehat{G}$ and $\widehat{H}$ be defined by \eqref{coeff-dare2}.
If the sequence $\{\widehat{Y}_k\}_{k=0}^\infty$ is generated by Algorithm~\ref{alg-afpi} with $\mathbf{A}_0 = \widehat{A}$, $\mathbf{G}_0 = \widehat{G}$, $\mathbf{H}_0 = \widehat{H}$ and $\widehat{Y}_0 = X_0 \in \mathcal{S}_\geq$, then
$\{\widehat{Y}_k\}_{k=0}^\infty$ is a nonincreasing sequence that converges at least superlinearly to the maimal solution $X_M$ of the CDARE \eqref{cdare} with the rate of convergence
\begin{equation*}  
 \limsup_{k\rightarrow \infty} \sqrt[r^k]{\|\widehat{Y}_k-X_M \|} \leq \rho (\widehat{T}_{X_M})^2,
\end{equation*}
provided that $\rho (\widehat{T}_{X_M})<1$.
\end{Theorem}
\begin{proof}
From the AFPI with $r>1$ defined by \eqref{afpi-r}, the sequence $\{\mathbf{X}_{k}\}_{k=0}^\infty$ generated by Algorithm~\ref{alg-afpi} satisfies the discrete flow property \eqref{DF}. Then we see that $\mathbf{A}_{k}={A}_{r^{k}-1}$, $\mathbf{G}_{k}={G}_{r^{k}-1}$, $\mathbf{H}_{k}={H}_{r^{k}-1}$ and
$\widehat{Y}_k = \widehat{\mathcal{R}}^{(r^k)} (X_0)$ for each $k\geq 1$, see, e.g., Remark 4.1 of \cite{l.c18},
where the operator $\widehat{\mathcal{R}}(\cdot )$ is defined by \eqref{dare2-b}.

As mentioned in the proof of Theorem~\ref{thm:inheri-1}, since $\{\widehat{Y}_k\}_{k=0}^\infty$ is a subsequence of $\{X_k\}_{k=0}^\infty$ generated by
the FPI \eqref{fpi}, it nonincreasingly converges to the maximal solution $X_M$ of the CDARE \eqref{cdare} as $k$ increases.
The rest of proof is straightforward from Theorem~\ref{thm:Xmax-dare} and Theorem~\ref{thm:inheri-2}.
\end{proof}

\section{Numerical examples} \label{sec:NumExamp}

In this section we shall give two examples to illustrate the correctness of aforementioned theorems in this paper
and the feasibility of the FPI \eqref{fpi} and the AFPI presented in Section~\ref{sec:afpi}.
In our numerical experiments shown below, we will measure the normalized residual
\[  \mbox{NRes} (Z) : = \frac{\|Z - \mathcal{R}(Z) \|}
{\|Z\| + \|A^H \overline{Z} A\| + \|A^H \overline{Z} B R_Z^{-1} B^H \overline{Z} A\| + \|H\|}  \]
for each quantity $Z$ computed by the FPI \eqref{fpi} or Algorithm~\ref{alg-afpi}, where $\|\cdot \|$ denotes the matrix $2$-norm.
 In our numerical results, each iterative method was terminated when $\mbox{NRes} \leq 1.0\times 10^{-15}$.

All numerical experiments were performed on ASUS laptop (ROG GL502VS-0111E7700HQ), using Microsoft Win-
dows 10 operating system and MATLAB Version R2019b, with Intel Core i7-7700HQ
CPU and 32 GB RAM.

\begin{example} \label{ex1}   \em
We first consider a scalar CDARE \eqref{cdare} of the form
\begin{equation} \color{black}
  x  = |a|^2 \bar{x} -\frac{|a|^2 \bar{x}^2 |b|^2}{r_0 + |b|^2\bar{x}} + h 
   = \frac{|a|^2 \bar{x}}{1+g\bar{x}} + h, \label{scdare}
\end{equation}
where $a,b\in \mathbb{C}$, $r_0,h\in \mathbb{R}$ with $r_0+|b|^2\bar{x} > 0$ and $g := |b|^2 / r_0$ with $r_0 \neq 0$. Without loss of generality, we assume that
$\color{black}  |a| > 0$, $r_0 >0$ and thus $g>0$.
From \eqref{scdare} and $\color{black} 1+g\bar{x} > 0$, we obtain
\[ g|x|^2 + x - (|a|^2 + gh)\bar{x} - h = 0,  \]
which has two solutions $\color{black} x_M, x_m \in \mathbb{H}_1 = \mathbb{R}$ satisfying
\begin{equation}  \label{xM}
x_M := \frac{-(1 - |a|^2 - gh) + \sqrt{D}}{2g}\quad \mbox{and}\quad  x_m := \frac{-(1 - |a|^2 - gh) - \sqrt{D}}{2g}
\end{equation}
if 
the discriminant
$D := (1 - |a|^2 - gh)^2 + 4gh \geq 0$. Note that
$D\geq 0$ if and only if $h\geq h_M:=\frac{-(1-|a|)^2}{g}$ or $h\leq h_m:=\frac{-(1+|a||)^2}{g}$.

For any $h > h_M$, let the coefficient matrices of the CDARE \eqref{cdare-a} be defined by
  \begin{align*}
  &A = \bb a & 0 & \cdots &0\\ 0 & 0 & \cdots & 0 \\ \vdots & \vdots & \ddots & \vdots \\ 0 & 0 & \cdots & 0 \eb\in \mathbb{C}^{n\times n},
  B = \bb b\\ 0 \\ \vdots \\ 0  \eb \in \mathbb{C}^{n\times 1}, R = r_0,\\
  &H = \bb h & h_{12} & \cdots & h_{1n} \\ \overline{h_{12}} & h_{22} & \cdots & h_{2n} \\ \vdots & \vdots & \ddots & \vdots \\
  \overline{h_{1n}} & \overline{h_{2n}} & \cdots & h_{nn} \eb\in \mathbb{H}_n,
   \end{align*}
where the entries $a$, $b$, $r_0$ and $h_{ij}$ are randomly generated by MATLAB coomands {\tt rand}, {\tt randn}, {\tt crand} and {\tt crandn}, respectively.
It is easily seen that the CDARE \eqref{cdare-a} has two extremal solutions, namely,
\[ X_M := \bb x_M & h_{12} & \cdots & h_{1n} \\ \overline{h_{12}} & h_{22} & \cdots & h_{2n} \\ \vdots & \vdots & \ddots & \vdots \\
  \overline{h_{1n}} & \overline{h_{2n}} & \cdots & h_{nn} \eb\in \mathbb{H}_n,\quad
  X_m := \bb x_m & h_{12} & \cdots & h_{1n} \\ \overline{h_{12}} & h_{22} & \cdots & h_{2n} \\ \vdots & \vdots & \ddots & \vdots \\
  \overline{h_{1n}} & \overline{h_{2n}} & \cdots & h_{nn} \eb\in \mathbb{H}_n,  \]
 with $x_M$ and $x_m$ being defined by \eqref{xM}. Moreover, the maximal  solution $X_M$ of the CDARE \eqref{cdare-a} satisfies $\rho (\widehat{T}_{X_M}) = \frac{|a|^2}{(1+gx_M)^2} < 1$.

 In our numerical experiments we tested the CDAREs with $n=500$. Starting with an initial matrix $X_0\in \mathcal{S}_\geq \cap \mathbb{T}$, the FPI \eqref{fpi} produced a highly assurate approximation $X_9$ to the maximal solution $X_M$ after $9$ iterations. In this case, its relative error is about $1.43\times 10^{-16}$
 and $\rho (\widehat{T}_{X_9}) \approx 0.072848 < 1$. The numerical results are reported in Table~\ref{tab1-ex1}.

\begin{table}[tbhp]
\begin{center}
\begin{tabular}{c|c|c}
  \hline
  $k$ & NRes($X_k$) & $\rho (\widehat{T}_{X_k})$  \\
 \hline\hline
 $1$ & $1.65\times 10^{-7}$ & $7.28\times 10^{-2}$   \\
  $2$ & $1.20\times 10^{-8}$ & $7.28\times 10^{-2}$   \\
   $3$ & $8.74\times 10^{-10}$ & $7.28\times 10^{-2}$   \\
   $4$ & $6.36\times 10^{-11}$ & $7.28\times 10^{-2}$   \\
   $5$ & $4.64\times 10^{-12}$ & $7.28\times 10^{-2}$   \\
   $6$ & $3.38\times 10^{-13}$ & $7.28\times 10^{-2}$   \\
   $7$ & $2.47\times 10^{-14}$ & $7.28\times 10^{-2}$   \\
   $8$ & $1.84\times 10^{-15}$ & $7.28\times 10^{-2}$   \\
   $9$ & $1.22\times 10^{-16}$ & $7.28\times 10^{-2}$   \\
  \hline
 \end{tabular}
\end{center}
\caption{Numerical results of the FPI \eqref{fpi} for Example \ref{ex1} with $n=500$.}\label{tab1-ex1}
\end{table}

   On the other hand, in order to verify the correctness of Theorem~\ref{thm:inheri-1} numerically, we
   transformed the CDARE \eqref{cdare-a} to the associated DARE \eqref{dare2}
 of coefficient matrices defined by \eqref{coeff-dare2}. Based on these matrices, we computed an approximation $\widetilde{Y}_M$ to
 the maximal solution of the DARE \eqref{dare2} via NATLAB built-in command {\tt dare}. Note that its relative error and
 the difference between $X_9$ and $\widetilde{Y}_M$ are
 \[ \frac{\| \widetilde{Y}_M - X_M \|}{\|X_M\|}\approx 9.04\times 10^{-16},\quad \|X_9 - \widetilde{Y}_M \| \approx 9.00\times 10^{-14},   \]
which give a numerical evidence for the inheritance property presented in Theorem~\ref{thm:inheri-1}.
In addition, applying the AFPI(2) performed by Algorithm~\ref{alg-afpi} with $\mathbf{A}_0 = \widehat{A}$,
$\mathbf{G}_0 = \widehat{G}$, $\mathbf{H}_0 = \widehat{H}$ and $\widehat{Y}_0 = X_0$, it generated an extremely accurate approximation $\widehat{Y}_3$ to $X_M$ with absolute error being $0.00\times 10^0$ after 3 iterations, and the numerical results are shown in Table~\ref{tab2-ex1}.
Again, the validity of Theorem~\ref{thm:inheri-1} follows from this numerical experiment.
\begin{table}[tbhp]
\begin{center}
\begin{tabular}{c|c|c}
  \hline
  $k$ & NRes($\widehat{Y}_k$) & $\rho (\widehat{T}_{\widehat{Y}_k})$  \\
 \hline\hline
 $1$ & $6.36\times 10^{-11}$ & $7.28\times 10^{-2}$   \\
  $2$ & $1.84\times 10^{-15}$ & $7.28\times 10^{-2}$   \\
   $3$ & $6.12\times 10^{-17}$ & $7.28\times 10^{-2}$   \\
   \hline
 \end{tabular}
\end{center}
\caption{Numerical results of AFPI(2) for Example \ref{ex1} with $n=500$.}\label{tab2-ex1}
\end{table}
\end{example}

\begin{example} \label{ex2} \em
  In this example we shall consider a critical case of the CDARE \eqref{cdare-a} with coefficient matrices
  defined as in Example~\ref{ex1}. For $h=h_M$, it follows from \eqref{xM} that the discriminant $D = 0$ and
  hence $x_M = x_m = \frac{|a|-1}{g}$. In this case, the spectrum of the matrix $\widehat{T}_{X_M}$ associated with the maximal solution $X_M$ is
  $\sigma (\widehat{T}_{X_M}) = \{1, 0\}$, where $\lambda = 1$ is a simple eigenvalue of $\widehat{T}_{X_M}$.

Although the upper bound of the convergence rate presented in Theorem~\ref{thm:Xmax} is not applicable here,
starting with a suitable matrix $X_0\in \mathcal{S}_\geq \cap \mathbb{T}$, the FPI \eqref{fpi} still converges to the maximal solution $X_M$ of the
CDARE theoretically, but the order of convergence might be linear or even sublinear.
With this reason, we apply the AFPI($r$) presented in Algorithm~\ref{alg-afpi} for computing the maximal solution of the CDARE \eqref{cdare} with
$n  = 500$ and different values of $r$. After the given CDARE is transformed into the DARE \eqref{dare2},
starting from $\mathbf{A}_0 = \widehat{A}$, $\mathbf{G}_0 = \widehat{G}$, $\mathbf{H}_0 = \widehat{H}$ and $\widehat{Y}_0 = X_0$,
the convergence histories of AFPI($r$) are shown in Figure~\ref{fig-ex2} for $r=2$, $5$, $9$ and $100$, respectively.
\begin{figure}[tbhp]
\centering
\includegraphics[height=10cm,width=12cm]{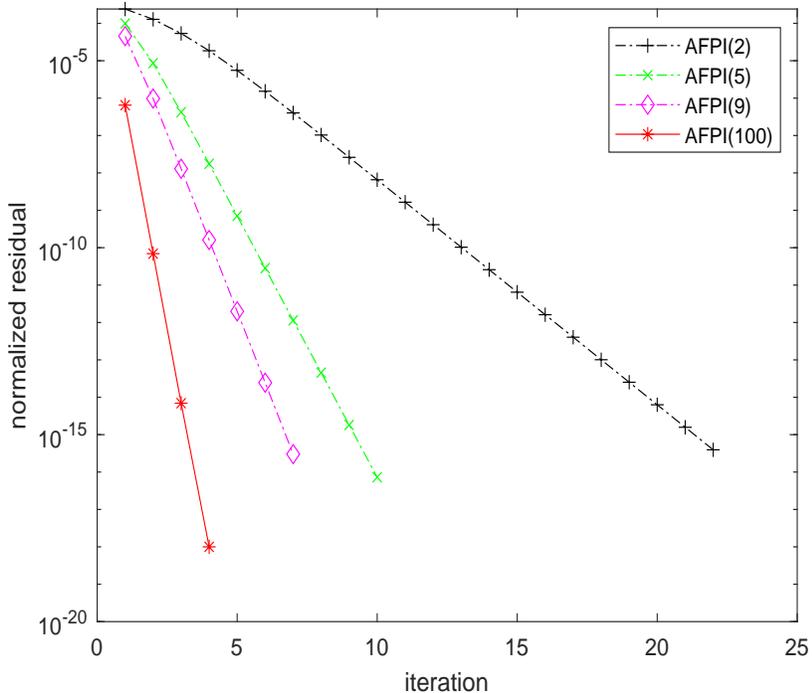}
\caption{Convergence histories of the AFPI for Example~\ref{ex2} with $n=500$.}
\label{fig-ex2}
\end{figure}
In particular, AFPI(100) produced an approximation $\widehat{Y}_4$ of $10$ significant digits
after $4$ iterations, since its relative error is about $4.89\times 10^{-10}$.
The numerical results of AFPI(100) are reported in Table~\ref{tab1-ex2}.
\begin{table}[tbhp]
\begin{center}
\begin{tabular}{c|c|c}
  \hline
  $k$ & NRes($\widehat{Y}_k$) & $\rho (\widehat{T}_{\widehat{Y}_k})$  \\
 \hline\hline
 $1$ & $6.49\times 10^{-7}$ & $9.90\times 10^{-1}$   \\
  $2$ & $6.93\times 10^{-11}$ & $1.00\times 10^{0}$   \\
   $3$ & $6.93\times 10^{-15}$ & $1.00\times 10^{0}$   \\
   $4$ & $9.99\times 10^{-19}$ & $1.00\times 10^{0}$   \\
   \hline
 \end{tabular}
\end{center}
\caption{Numerical results of AFPI(100) for Example \ref{ex2} with $n=500$.}\label{tab1-ex2}
\end{table}

Note that the rate of convergence of the AFPI (or FPI) presented in Theorem~\ref{thm:afpi-conv}
(or Theorem~\ref{thm:Xmax}) is no longer applicable for this example  because $\rho (\widehat{T}_{X_M}) = 1$.
We will investigate these interesting issues in our future work.
\end{example}

\section{Concluding remarks} \label{sec:cr}

In this paper we mainly deal with a class of conjugate discrete-time Riccati equations,
arising originally from the LQR control problem for discrete-time antilinear systems. In this case, the design of the optimal controller usually depends on the existence of a unique positive semidefinite optimizing solution of CDAREs \eqref{cdare-a} with $R>0$ and $H\geq 0$, if the antilinear system is assumed to be controllable.
Recently, under the assumptions in \eqref{ma}, it is shown in \cite{f.c23} that the existence of the maximal Hermitian solution can be extended to the CDARE \eqref{cdare-a} with $R$ being nonsingular and $H\in \mathbb{H}_n$, respectively, which is summarized in Theorem~\ref{thm:Xmax}.

From Theorem~\ref{thm:Xmax}, Theorem~\ref{thm:Xmax-dare} and Lemma~\ref{lem:inheri}, we have deduced that
the CDARE \eqref{cdare} and its transformed DARE \eqref{dare2} both have the same
maximal Hermitian solution $X_M$ in Theorem~\ref{thm:inheri-1}. As a special case mentioned in Theorem~\ref{thm:inheri-2},
their corresponding closed-loop matrices related to $X_M$ coincide as well.
To our best knowledge, these inheritance properties of the CDARE \eqref{cdare} have not been shown in the existing literatures.

From the computational point of view, inspired by Theorem~\ref{thm:inheri-1}, the AFPI presented in Algorithm~\ref{alg-afpi} is proposed for computing $X_M$ via the transformed DARE \eqref{dare2} directly. Numerical results in Section~\ref{sec:NumExamp} show the correctness of Theorem~\ref{thm:inheri-1} and the feasibility of the AFPI for finding the maximal solution of CDAREs numerically. In particular, although the convergence rate of the AFPI presented in Theorem~\ref{thm:afpi-conv} is no longer applicable for Example~\ref{ex2}, it seems that our AFPI algorithm still performed well when the CDARE \eqref{cdare} has
a maximal and almost stabilizing solution,and we shall further investigate this critical case in our futute work.

\section*{Acknowledgment}

This research work is partially supported by the National Science and Technology Council of Taiwan and
the National Center for Theoretical Sciences of Taiwan. The first author (Chun-Yueh Chiang) would like to thank the support from the National Science and Technology Council of Taiwan under the grant MOST 111-2115-M-150-001-MY2, and the corresponding author (Hung-Yuan Fan) would like to
thank the support from the National Science and Technology Council of Taiwan under the grant MOST 111-2115-M-003-012.


\end{document}